\documentclass[a4paper,12pt]{article}
\usepackage[english]{babel}
\usepackage[utf8]{inputenc}
\usepackage{hyperref}
\usepackage{geometry}

\usepackage{theoremref}
\usepackage{authblk}
\usepackage{amsmath}
\usepackage{amsfonts}
\usepackage{amssymb}
\usepackage{graphicx}
\usepackage{subcaption}
\usepackage{float}
\usepackage[nottoc,numbib]{tocbibind}
\usepackage{amsthm}	
\usepackage{enumerate}

\title{Fractional cross intersecting families}
\author[1]{Rogers Mathew}
\author[2]{Ritabrata Ray}
\author[3]{Shashank Srivastava}
\affil[1]{Department of Computer Science and Engineering,\protect\\ Indian Institute of Technology Kharagpur, Kharagpur 721302, India,
\protect\\ rogersmathew@gmail.com}
\affil[2]{Department of Electronics and Electrical Communication Engineering, \protect\\ Indian Institute of Technology Kharagpur, Kharagpur 721302, India,\protect\\ rayritabrata96@gmail.com}
\affil[3]{Toyota Technological Institute at Chicago, Chicago 60615, USA, \protect\\ shashanksri47@gmail.com}

\date{\vspace{-5ex}}

\theoremstyle{plain}
\newtheorem{theorem}{Theorem}[section]
\newtheorem{lemma}[theorem]{Lemma}
\newtheorem{proposition}[theorem]{Proposition}
\newtheorem{corollary}[theorem]{Corollary}
\newtheorem{claim}[theorem]{Claim}

\newtheorem{definition}{Definition}
\newtheorem{observation}{Observation}
\newtheorem{construction}{Construction}

\begin{document}

 \maketitle
 \begin{abstract}
Let $\mathcal{A}=\{A_{1},...,A_{p}\}$ and $\mathcal{B}=\{B_{1},...,B_{q}\}$ be two families of subsets of $[n]$ such that for every $i\in [p]$ and $j\in [q]$, $|A_{i}\cap B_{j}|= \frac{c}{d}|B_{j}|$, where $\frac{c}{d}\in [0,1]$ is an irreducible fraction. We call such families \emph{$\frac{c}{d}$-cross intersecting families}. In this paper, we find a tight upper bound for the product  $|\mathcal{A}||\mathcal{B}|$ and characterize the cases when this bound is achieved for $\frac{c}{d}=\frac{1}{2}$. Also, we find a tight upper bound on $|\mathcal{A}||\mathcal{B}|$ when $\mathcal{B}$ is $k$-uniform and characterize, for all $\frac{c}{d}$, the cases when this bound is achieved.
\end{abstract}
 \begin{section}{Introduction}
 Let $[n]$ denote $\{1,...,n\}$ and let $2^{[n]}$ denote the power set of $[n]$.We shall use $\binom{[n]}{k}$ to denote the set of all $k$-sized subsets of $[n]$. Let $\mathcal{F} \subseteq 2^{[n]}$. The family $\mathcal{F}$ is an \emph{intersecting family} if every two sets in $\mathcal{F}$ intersect with each other. The famous Erd\H{o}s-Ko-Rado Theorem \cite{10.1093/qmath/12.1.313} states that $|\mathcal{F}|\leq \binom{n-1}{k-1}$ if $\mathcal{F}$ is a $k$-uniform intersecting family, where $2k\leq n$. Several variants of the notion of intersecting families have been extensively studied in the literature. Given a set $L=\{ l_{1},\ldots,l_{s} \}$ of non-negative integers, a family $\mathcal{F} \subseteq 2^{[n]}$ is \emph{$L$-intersecting} if for all $F_{i},F_{j}\in \mathcal{F}, F_{i}\neq F_{j}, |F_{i} \cap F_{j}| \in L$. Ray-Chaudhuri and Wilson in \cite{ray-chaudhuri1975} showed that if $\mathcal{F}$ is $k$-uniform and $L$-intersecting, then $|\mathcal{F}|\leq \binom{n}{s}$ and the bound is tight. Frankl and Wilson in \cite{article1} showed a tight upper bound of $\binom{n}{s} + \binom{n}{s-1} + \cdots + \binom{n}{0}$ if the restriction on the cardinalities of the sets of an $L$-intersecting family is relaxed. Further, if $L$ is a singleton set, then Fisher inequality \cite{bose1949note} gives an upper bound of $|\mathcal{F}|\leq n$ for the cardinality of an $L$-intersecting family $\mathcal{F}$. Recently, in \cite{DBLP:journals/corr/abs-1803-03954}, Balachandran et al. introduced a fractional variant of the classical $L$-intersecting families. For a survey on intersecting families, see \cite{article}. \par
 Two families $\mathcal{A}, \mathcal{B} \subseteq 2^{[n]}$ are \emph{cross-intersecting} if $|A\cap B|>0$, $\forall$ $A \in \mathcal{A}$,$ B \in \mathcal{B}$. Pyber in \cite{PYBER198685} showed that if $n\geq 2k$, and $\mathcal{A},\mathcal{B} \subseteq \binom{[n]}{k}$ is a cross-intersecting pair of families, then $|\mathcal{A}||\mathcal{B}|\leq \binom{n-1}{k-1}^2$. Frankl et al. in \cite{FRANKL2014207} showed that if $\mathcal{A},\mathcal{B} \subset \binom{[n]}{k}$ such that $|A \cap B| \geq t$ for all $A \in \mathcal{A}$ and $ B \in \mathcal{B}$, then for all $n \geq (t+1)(k-t+1)$, $|\mathcal{A}||\mathcal{B}|\leq \binom{n-t}{k-t}^2$, the cross-intersecting version of the  Erd\H{o}s-Ko-Rado Theorem. A cross-intersecting pair of families $\mathcal{A},\mathcal{B}\subseteq 2^{[n]}$ is said to be $l$-cross-intersecting if $\forall A \in \mathcal{A},~ B \in \mathcal{B}$, $|A\cap B|=l$, for some positive integer $l$. 
 Ahlswede, Cai and Zhang showed in \cite{AHLSWEDE198975}, for all $n \geq 2l$, a simple construction of an $l$-cross-intersecting pair $(\mathcal{A},\mathcal{B})$ of families of subsets of $[n]$ with $|\mathcal{A}||\mathcal{B}|=\binom{2l}{l}2^{n-2l}=\Theta(\frac{2^n}{\sqrt l})$. Later Alon and Lubetzky in \cite{Alon2009} showed that the $\Theta(\frac{2^n}{\sqrt l})$ bound is tight and characterized the cases when the bound is achieved.\par
 In this paper, we introduce a fractional variant of the $l$-cross-intersecting families.
 Let $\mathcal{A}=\{A_{1},...,A_{p}\}$ and $\mathcal{B}=\{B_{1},...,B_{q}\}$ be two families of subsets of $[n]$ such that for every $i\in [p]$ and $j\in [q]$, $|A_{i}\cap B_{j}|= \frac{c}{d}|B_{j}|$, where $\frac{c}{d}\in [0,1]$ is an irreducible fraction. We call such an $(\mathcal{A},\mathcal{B})$ pair a $\frac{c}{d}$-cross-intersecting pair of families. Given $c$, $d$, and $n$, let $\mathcal{M}_{\frac{c}{d}}(n)$ denote the maximum value of $|\mathcal{A}||\mathcal{B}|$ where $(\mathcal{A},\mathcal{B})$ is a $\frac{c}{d}$-cross intersecting pair of families of subsets of $[n]$.
 We have the following results:

\begin{theorem}\thlabel{thm:1.1}
 $\mathcal{M}_{\frac{c}{d}}(n) = 2^n$
\end{theorem}
When $\frac{c}{d}=0$, $\mathcal{A}=2^{[n]}$, $\mathcal{B}=\{\emptyset\}$ is a maximal pair. In fact, $\mathcal{A}=2^{[k]}$, $\mathcal{B}=\mathcal{P}(S)$, where $\mathcal{P}(S)$ is the power set of $S=\{k+1,\ldots,n\}$, are the only maximal pairs up to a relabelling of the elements, $0\leq k \leq n$.
When $\frac{c}{d}=1$, $\mathcal{A}=\{[n]\}$ and $\mathcal{B}=2^{[n]}$ is a maximal pair. In fact, $\mathcal{B}=2^{[k]}$, $\mathcal{A}=\{A: A=[k]\cup T$, where $ T \in \mathcal{P}(S)\}$, where $\mathcal{P}(S)$ is the power set of $S=\{k+1,\ldots,n\}$, are the only maximal pairs up to a relabelling of the elements, $0\leq k \leq n$.
In \thref{thm:1.2}, we characterize all maximal pairs when $\frac{c}{d}=\frac{1}{2}$.
\begin{theorem}\thlabel{thm:1.2}
 Let $(\mathcal{A},\mathcal{B})$ be a $\frac{1}{2}$-cross intersecting pair of families of subsets of $[n]$ with $|\mathcal{A}||\mathcal{B}|=2^n$. Then $(\mathcal{A},\mathcal{B})$ is one of the following $\lfloor\frac{n}{2}\rfloor +1$ pairs of families $(\mathcal{A}_{k},\mathcal{B}_{k})$, $0 \leq k \leq \lfloor \frac{n}{2} \rfloor$, up to isomorphism.
 \begin{center}
     $\mathcal{A}_{0}=2^{[n]}$ and $\mathcal{B}_{0}=\{\emptyset\}$
 \end{center}
 \begin{center}
     $\mathcal{A}_{k}=\{ A \in 2^{[n]} : |A \cap \{ 2i-1, 2i\}|= 1 ~~\forall i, 1\leq i\leq k\}$
 \end{center}
 \begin{center}
     $\mathcal{B}_{k}=\{ B \in 2^{[n]} : |B \cap \{ 2i-1, 2i \}| \in \{ 0, 2\}~~ \forall i, 1\leq i\leq k $ and  $\forall j > 2k$, $j \notin B\}$,
     \end{center}
    where $1\leq k \leq \lfloor \frac{n}{2} \rfloor$.
\end{theorem}

It would be interesting to show a characterization theorem for any $\frac{c}{d}\in [0,1]$. We do have such a general characterization theorem (along with a new tight upper bound) in \thref{thm:1.3} for the case when $\mathcal{B}$ is $k$-uniform. The proof is a direct application of Theorem 1.1 in \cite{Alon2009}.
\begin{theorem}\thlabel{thm:1.3}
 Let $(\mathcal{A}$,$\mathcal{B})$ be a $\frac{c}{d}$-cross intersecting pair of families of subsets of $[n]$. Let $\mathcal{B}$ be $k$-uniform. Then, there exists some $k_{0}>0$, such that for $k>k_{0}$ we have 
 \begin{center}
     $|\mathcal{A}||\mathcal{B}|\leq \binom{\frac{2ck}{d}}{\frac{ck}{d}}2^{n-\frac{2ck}{d}}$
 \end{center}
 and the bound is tight if and only if, either $(a)$ or $(b)$ hold:
 \begin{center}
 \begin{enumerate}[(a)]
     \item  When $\frac{c}{d}=1$, $\mathcal{A}=\{\{1,\ldots,\kappa\}\}\times2^Y$, $\mathcal{B}=\binom{[\kappa]}{k}$ where $Y=\{\kappa+1,\ldots,n\}$ and $\kappa \in \{ 2k-1, 2k\}$ up to a relabelling of the elements of $[n]$.

 \item When $\frac{c}{d}\neq 1$:
 
     \begin{enumerate}[(i)]
     \item If $k$ is even, $c=1$, $d=2$, $\frac{ck}{d}=\lceil\frac{k}{2}\rceil$,
     \item If $k$ is odd, $c=\frac{k+1}{2}$, $d=k$, $\frac{ck}{d}=\lceil\frac{k}{2}\rceil$, 
      \end{enumerate}
\end{enumerate}
 \end{center}
 and for both the cases($(i)$ and $(ii)$), 
 there exists some $\tau$ such that, 
 $k+\tau\leq n$ 
and up to a relabelling of the elements of $[n]$,
\begin{center}
    $\mathcal{A}=
    \{ \cup_{T\in J} \, {T} : J \subset
    \{\{1,k+1\},\ldots,\{\tau,
    k+\tau\},\{\tau+1\},\ldots,\{k\}\},|J|=\lceil\frac{k}{2}\rceil\}
    \times 2^{X}$
\end{center}
where $X=\{k+\tau+1,\ldots,n\}$ and
\begin{center}
$\mathcal{B}=\{L\cup\{\tau+1,\ldots, k\}:L\subset \{1,\ldots,\tau,k+1,\ldots,k+\tau\}, |L\cap\{i,k+i\}|=1$ for all $i \in [\tau] \}$.
    
\end{center}

\end{theorem}

\end{section}
  
\begin{section}{Notations and definitions}
Given any $S \subseteq [n]$, we shall use  $\chi(S)$ to denote the \textit{characteristic vector} of $S$ which is a $0-1$ vector of size $n$ having its $i^{th}$ entry equal to $1$ if and only if $i \in S$. The \textit{weight} of a vector is the number of non-zero entries it has, and hence weight of $\chi(S)$ is the same as $|S|$. \par

For any family $\mathcal{A} \subseteq 2^{[n]}$, we shall (ab)use $\mathcal{A}$ to denote the collection of characteristic vectors of the members of $\mathcal{A}$ as well. The meaning will be clearly stated if not clear from the context.\par
Let $V$ be a collection of vectors in $\mathbb{F}^{n}_{2}$. Then, we define the following: 
\begin{enumerate}
    \item $span(V)$: The collection of all the vectors that can be expressed as a linear combination in $\mathbb{F}_2$ of the vectors of $V$. We know that $span(V)$ is a vector space over $\mathbb{F}_{2}$.
    \item $basis(V)$: We use $basis(V)$ to denote the basis of $span(V)$.
    \item $dim(V)$: $dim(V)=|basis(V)|$
\end{enumerate}
\begin{definition}
 $V\subseteq \mathbb{F}^{n}_{2} $ is a \emph{linear code} if $V=span(V)$.
\end{definition}
\begin{definition}
 Given a linear code $C\subseteq \mathbb{F}^{n}_{2}$, the \emph{dual code} $C^{\perp}$ is defined as,
 \begin{center}
   $C^{\perp}=\{ x \in \mathbb{F}^{n}_{2} | \langle x,c \rangle =0, \forall c \in C\}$  
 \end{center}
 where $\langle x, y\rangle$ is the standard inner product over $\mathbb{F}_{2}$.
\end{definition}
The following is a well-known fact that is easy to verify.
\begin{lemma}
 If $C \subseteq \mathbb{F}^{n}_{2}$ is a linear code, then $C^{\perp}$ is also a linear code.
\end{lemma}
  \begin{definition}
   Self orthogonal and self dual codes:
   A code $C$ is \textit{self orthogonal} if $C\subseteq C^{\perp} $
   and it is \textit{self dual} if $C=C^{\perp}$.
  \end{definition}
 \end{section}
 \begin{section}{Bounding $\mathcal{M}_{\frac{c}{d}}(n)$}
 Let $(\mathcal{A},\mathcal{B})$ be a  $\frac{c}{d}$-cross-intersecting pair of families of subsets of $[n]$, where $\frac{c}{d} \in [0,1]$ is an irreducible fraction. We shall (ab)use $\mathcal{A},\mathcal{B}$ to denote the set of characteristic vectors of the sets in $\mathcal{A},\mathcal{B}$ respectively.
 For any $a \in \mathcal{A}, b \in \mathcal{B}$, we observe that $\langle a, b \rangle \equiv |A\cap B| ~(\mathrm{mod}~ 2)$, where $a = \chi (A)$, $b = \chi (B)$.\par
 Partition the family $\mathcal{B}$ into two parts as,
 \begin{center}
 \begin{gather}
     \mathcal{B}_{1}=\{B\in \mathcal{B} : |B| \equiv 0 ~(\mathrm{mod}~ 2d)\}\\
     \mathcal{B}_{2}=\{B\in \mathcal{B} : |B| \equiv d ~(\mathrm{mod}~ 2d)\}
 \end{gather}
 \end{center}
   
 As all the sets $B\in \mathcal{B}$ have their cardinality $|B|$ divisible by $d$, $\{\mathcal{B}_{1}$,$\mathcal{B}_{2}\}$ is a valid partition of $\mathcal{B}$.
 Therefore  $\forall a \in \mathcal{A}$ , $b \in \mathcal{B}$, using the $\frac{c}{d}$ intersection property, we have:
 \begin{center}
 \begin{gather*}
     \langle a, b \rangle= 
     \begin{cases}
     1, if \: b \in \mathcal{B}_{2} \:and\: c \:is\: odd \\
     0, \: otherwise
      \end{cases}
 \end{gather*}

 \end{center}
 \begin{construction}
 Construct a set $\mathcal{B}^{'}_{1}$, by appending  a $0$ to the left of every vector in $\mathcal{B}_{1}$, and a set $\mathcal{B}^{'}_{2}$ by appending a 1 to the left of every vector in $\mathcal{B}_{2}$. Let  $\mathcal{B}^{'}=\mathcal{B}^{'}_{1} \cup \mathcal{B}^{'}_{2}$.
 Construct a set $\mathcal{A}^{'}$ by appending a $1$ to the left of every vector in $\mathcal{A}$. 
 \end{construction}
 We now have, the value of
 
  \begin{center}
      $\langle a,b\rangle=0 ~~ \forall a \in \mathcal{A}^{'}$, $b \in \mathcal{B}^{'}$
  \end{center}
  So, $(span(\mathcal{A}^{'}),span(\mathcal{B}^{'}))$ is a pair of mutually orthogonal subspaces of $\mathbb{F}^{n+1}_{2}$ over $\mathbb{F}_{2}$. We thus have, 
  \begin{center}
      $dim(span(\mathcal{A}^{'})) + dim(span(\mathcal{B}^{'}))\leq n+1$
  \end{center}
  So, it follows that
   \begin{equation}\label{eq:1}
  \begin{split}
    |\text{span}(\mathcal{A}^{'})|\cdot|\text{span}(\mathcal{B}^{'})| & = 2^{\text{dim}(span(\mathcal{A}^{'}))}\cdot2^{\text{dim}((span(\mathcal{B}^{'}))}\\
    & = 2^{ dim(span(\mathcal{A}^{'})) + dim(span(\mathcal{B}^{'}))}\\
    & \le 2^{n+1}
  \end{split}
\end{equation}
\begin{lemma}\thlabel{lem:3.1}
 If the elements of a linear code $C \subseteq \mathbb{F}^{n}_{2} $ are arranged as rows of a matrix $M_{C}$ with $n$ columns, then for each column, one of the following holds,
 \begin{enumerate}[(i)]
 \item  All the entries in that column are $0$
 \item  Exactly half the entries in that column are $0$, and the rest are $1$.
 \end{enumerate}

\end{lemma}
\begin{proof}
     As $C$ is a linear code, if we pick any $a \in C$, and consider the set $S=\{a+x | x\in C\}$ where $a+x$ is the vector addition in $\mathbb{F}^{n}_{2}$, then by the definition of a linear code $S=C$. Let $M_{S}$ be a matrix whose rows are the vectors of $S$, taken in any order. $M_{S}$ and $M_{C}$ have the same set of rows (only their order may differ). \par
   Let $j\in [n]$. Column $j$ in $M_{C}$ and $M_{S}$ have the same number of $1$'s( and $0$'s). Suppose (i) does not hold for column $j$ in $M_{C}$. Then, some row, say $a$, in $M_{C}$ has its $j^{th}$ entry as 1. Let $S$, and thereby $M_{S}$, be defined according to this vector $a$. From the definition of $S$, it is clear that the number of $1$'s in the $j^{th}$ column of $M_{S}$ is equal to the number of $1$'s in the $j^{th}$ column of $M_{C}$. Since adding $a$ to any $\{0,1\}$ vector flips the $j^{th}$ coordinate of $v$, we conclude that $(ii)$ holds for $M_{c}$. 
\end{proof}   
\begin{corollary}\thlabel{cor:3.2}
 $|span(\mathcal{A}^{'})|\geq 2 |\mathcal{A}^{'}|$
\end{corollary}
\begin{proof}
     The leftmost column of $\mathcal{M}_{\mathcal{A}^{'}}$ does not contain any $0$. As $span(\mathcal{A}^{'})$ is a linear code and $\mathcal{A}^{'}\subseteq span(\mathcal{A}^{'})$, by condition (ii) of \thref{lem:3.1} above, $span(\mathcal{A}^{'})$ must have at least $|\mathcal{A}^{'}|$ more elements having their leftmost entry as $0$.
\end{proof}
Now we prove the main result of this section which is \thref{thm:1.1}.  \vspace{0.1in}  \\
\textbf{Statement of \thref{thm:1.1}:} $\mathcal{M}_{\frac{c}{d}}(n) = 2^n$
\begin{proof}
     $\mathcal{A}=2^{[n]}$, $\mathcal{B}=\{\emptyset\}$ is a trivial example of a $\frac{c}{d}$ cross-intersecting pair of families having $|\mathcal{A}||\mathcal{B}|=2^n$. Thus, $\mathcal{M}_{\frac{c}{d}}(n) \geq 2^n$.
     The proof of the upper bound for $\mathcal{M}_{\frac{c}{d}}(n)$ follows from Inequality \eqref{eq:1} and \thref{cor:3.2}. Let $(\mathcal{A},\mathcal{B})$ be a $\frac{c}{d}$ cross-intersecting pair of families of subsets of $[n]$. Let $\mathcal{A}^{'}$, $\mathcal{B}^{'}$ be constructed from $\mathcal{A}$, $\mathcal{B}$, respectively, as explained in the beginning of this section. Note that $|\mathcal{A}^{'}| = |\mathcal{A}|$ and $|\mathcal{B}^{'}| = |\mathcal{B}|$ by construction.
 \begin{align*}
   2^{n+1} & \ge |\text{span}(\mathcal{A}^{'})|\cdot|\text{span}(\mathcal{B}^{'})| && \text{\hfill [from \eqref{eq:1}]}\\
	   & \ge 2\cdot|\mathcal{A}^{'}|\cdot|\text{span}(\mathcal{B}^{'})| && \text{\hfill [from \thref{cor:3.2}]}\\
	   & \ge 2\cdot|\mathcal{A}^{'}|\cdot|\mathcal{B}^{'}| \\
	   & = 2\cdot|\mathcal{A}|\cdot|\mathcal{B}| && \text{[by construction]}\\
 \end{align*}

\end{proof}
\end{section}

\begin{section}{Characterization of maximal pairs when $\frac{c}{d}=\frac{1}{2}$}

\begin{definition}
 Cross bisecting pair of families: A pair of families of subsets of $[n]$ is called a \emph{cross-bisecting pair} if it is a  $\frac{1}{2}$ cross-intersecting pair. $(\mathcal{A},\mathcal{B})$ is called a \emph{maximal} cross bisecting or simply a \emph{maximal pair}, if it is a cross bisecting pair and $|\mathcal{A}||\mathcal{B}|=2^n$.
\end{definition}
For example, $\mathcal{A}=2^{[n]}$ and $\mathcal{B}=\{\emptyset\}$ is a trivial maximal pair. In this section, we characterize all maximal pairs.
Let $(\mathcal{A},\mathcal{B})$ be a cross bisecting pair and let $(\mathcal{A}^{'},\mathcal{B}^{'})$ be the associated pair constructed by appending bits as defined in the previous section.
\begin{definition}
 Let $f_{\mathcal{A}} : \mathcal{A} \rightarrow \mathcal{A}^{'}$ be a bijective mapping that maps every vector in $\mathcal{A}$ to its corresponding vector in $\mathcal{A}^{'}$, and let $g_{\mathcal{A}} : \mathcal{A}^{'} \rightarrow \mathcal{A}$ be its inverse. Likewise, define functions $f_{\mathcal{B}}$ and $g_{\mathcal{B}}$ between $\mathcal{B}$ and $\mathcal{B}^{'}$. For any set $V \subseteq \mathcal{A}$, we shall use, $f_{\mathcal{A}}(V)$ to denote $\{f_{\mathcal{A}}(A) |\ A \in V\}$ and for any $V\subseteq \mathcal{A}^{'}$, we use $g_{\mathcal{A}}(V)$ to denote  $\{g_{\mathcal{A}}(A) |\ A \in V\}$.
 Similarly, for any  $V \subseteq \mathcal{B}$, we use, $f_{\mathcal{B}}(V)$ to denote $\{f_{\mathcal{B}}(B) |\ B \in V\}$ and for any $V\subseteq \mathcal{B}^{'}$, $g_{\mathcal{B}}(V)$ to denote $\{g_{\mathcal{B}}(B) |\ B \in V\}$ 
 \end{definition}

\begin{observation}\thlabel{obv}
$f_{\mathcal{B}}(\mathcal{B}_{1})=\mathcal{B}_{1}^{'}$ and $f_{\mathcal{B}}(\mathcal{B}_{2})=\mathcal{B}_{2}^{'}$. 
Similarly, $g_{\mathcal{B}}(\mathcal{B}_{1}^{'})=\mathcal{B}_{1}$ and $g_{\mathcal{B}}(\mathcal{B}_{2}^{'})=\mathcal{B}_{2}$\\
\end{observation}
Suppose $(\mathcal{A},\mathcal{B})$ is a maximal pair. Then from the proof of \thref{thm:1.1}, we must have :

\begin{align}
 |\text{span}(\mathcal{A}^{'})| & = 2|\mathcal{A}^{'}|\\
 |\text{span}(\mathcal{B}^{'})| & = |\mathcal{B}^{'}| \label{eq:3}\\ 
 \text{dim}(span(\mathcal{A}^{'})) + \text{dim}(span(\mathcal{B}^{'})) & = n+1
\end{align}
\begin{proposition}\thlabel{lem:4.3}
$\mathcal{B} = \text{span}(\mathcal{B})$. Further, $f_{\mathcal{B}}$ is a linear map.
\end{proposition}
\begin{proof}
     This follows from equation \eqref{eq:3}. Let $x_1, x_2 \in \mathcal{B}$. We show that $x_3 = x_1 + x_2 \in \mathcal{B}$. This would imply $\mathcal{B}$ is closed under addition in $\mathbb{F}^{n}_2$ over $\mathbb{F}_2$, and hence $\mathcal{B} = \text{span}(\mathcal{B})$.\par
     Let $x^{'}_1 = f_{\mathcal{B}}(x_1)$ and $x^{'}_2 = f_{\mathcal{B}}(x_2)$. From Equation \eqref{eq:3}, we have, $w = x^{'}_1 + x^{'}_2 \in \mathcal{B}^{'}$. Since $w$ and $x_3$ agree on each of the rightmost $n$ bits of $x_3$, we have $g_{\mathcal{B}}(w)=x_3$. Since $w\in \mathcal{B}^{'}$, from the definition of the function $g_{\mathcal{B}}$ we have $x_3=g_{\mathcal{B}}(w)\in \mathcal{B}$. Further, observe that $f_{\mathcal{B}}(x_1) + f_{\mathcal{B}}(x_2)=$ $w=f_{\mathcal{B}}(x_3)=f_{\mathcal{B}}(x_1+x_2)$ and hence $f_{\mathcal{B}}$ is a linear map. 
\end{proof}
That $\mathcal{B}$ is a linear code from \thref{lem:4.3} implies closure of the family of subsets $\mathcal{B}$ under symmetric difference. In fact, we have the following stronger result.
\begin{proposition}\thlabel{lem:4.4}
Let vectors $b_1,b_2 \in \mathcal{B}$. Then, $b_1 + b_2 \in \mathcal{B}_1$ if and only if either $b_1$,$b_2$ $\in \mathcal{B}_1$, or $b_1$,$b_2$ $\in \mathcal{B}_2$. Otherwise, $b_1 + b_2 \in \mathcal{B}_2$.
\end{proposition}
\begin{proof}
      We prove the 2-way implication, and rest of the proposition follows from \thref{lem:4.3}. Let $b^{'}_1 = f_{\mathcal{B}}(b_1), b^{'}_2 = f_{\mathcal{B}}(b_2)$.
 \begin{itemize}
  \item $b_1 + b_2 \in \mathcal{B}_1 \Rightarrow b_1 \text{ and } b_2$ are both from $\mathcal{B}_1$, or both from $\mathcal{B}_2$ \\
   Since $f_{\mathcal{B}}$ is a linear map, we have  $ (b_1 + b_2 \in \mathcal{B}_1) \Rightarrow (f_{\mathcal{B}}(b_1+b_2)=f_{\mathcal{B}}(b_1)+f_{\mathcal{B}}(b_2)=b^{'}_{1}+b^{'}_{2} \in \mathcal{B}^{'}_1)$. So, the leftmost bit of $b^{'}_1 + b^{'}_2$ is $0$. This means that the leftmost bit must be the same in $b^{'}_1$ and $b^{'}_2$, which directly implies that either $b_1^{'}$,$b_2^{'}$ $\in \mathcal{B}_1^{'}$, or $b_1^{'}$,$b_2^{'}$ $\in \mathcal{B}_2^{'}$. 
  
  \item Either $b_1$,$b_2$ $\in \mathcal{B}_1$, or $b_1$,$b_2$ $\in \mathcal{B}_2$ $\Rightarrow b_1 + b_2 \in \mathcal{B}_1$\\
   Since $b_1^{'}$ and $b_2^{'}$ agree upon the leftmost bit,  $b^{'}_1 +
   b^{'}_2$ has a $0$ in its leftmost bit. So, $b^{'}_1 + b^{'}_2 \in
   \mathcal{B}_{1}^{'}$. From the \thref{obv} above, we have $b_1 + b_2
   \in \mathcal{B}_1$.  
 \end{itemize}
\end{proof}
\begin{proposition}\thlabel{lem:4.5}
 $\mathcal{B}$ is a self-orthogonal code.
\end{proposition}
\begin{proof}
  We prove the proposition by showing that $\forall b_1,b_2 \in \mathcal{B}$, $\langle b_1, b_2\rangle=0$. Let $B_1,B_2$ be the sets corresponding to the vectors $b_1,b_2$, respectively. Since we are operating in the field $\mathbb{F}_{2}$, it is enough to show that $|B_{1}\cap B_{2}|$ is even.
 
  Let $b_3=b_{1}+b_{2}$. We observe that $b_3$ is the characteristic vector of $B_{3}=B_{1}\Delta B_{2}$, the symmetric difference of $B_{1}$ and $B_{2}$.
  We have,
  
   \begin{align}
   |B_{3}|=|B_{1}\Delta B_{2}|=|B_{1}|+|B_{2}|-2|B_{1}\cap B_{2}| \label{eq:5}
   \end{align}

   As $\frac{c}{d}=\frac{1}{2}$, $\forall B \in \mathcal{B}_{1}$, we have $|B|\equiv 0~(\mathrm{mod}~4)$. 
   By \thref{lem:4.3}, $B_{1}\Delta B_{2}=B_{3}\in \mathcal{B}$ as $\mathcal{B}$ is a linear code.
   Taking equation \eqref{eq:5} modulo 4, if $B_{3}\in \mathcal{B}_{1}$, then
   \begin{equation*}
       |B_{1}|+|B_{2}|-2|B_{1}\cap B_{2}|\equiv 0 ~(\mathrm{mod}~ 4)
   \end{equation*}
   By \thref{lem:4.4}, both $B_{1}$ and $B_{2}$ are either from $\mathcal{B}_{1}$ or from $\mathcal{B}_{2}$. In both cases, $|B_{1}|+|B_{2}|\equiv 0 ~(\mathrm{mod}~ 4)$
   Therefore, $2|B_{1}\cap B_{2}|\equiv 0 ~(\mathrm{mod}~ 4)$ or
           $|B_{1}\cap B_{2}|\equiv 0 ~(\mathrm{mod}~ 2)$.
           If $B_{3}\in \mathcal{B}_{2}$, then 
           \begin{equation*}
             |B_{1}|+|B_{2}|-2|B_{1}\cap B_{2}| \equiv |B_{3}| \equiv 2 ~(\mathrm{mod}~ 4)
           \end{equation*}
           
 Again by \thref{lem:4.4}, $|B_{1}|+|B_{2}|\equiv 2 ~(\mathrm{mod}~ 4)$.\\
 So, we have $2|B_{1}\cap B_{2}|\equiv 0 ~(\mathrm{mod}~ 4)$ or
           $|B_{1}\cap B_{2}|\equiv 0 ~(\mathrm{mod}~ 2)$.
 Thus in both cases, $|B_{1}\cap B_{2}|$
is even, so $\mathcal{B}$ is a self-othogonal code.   
  \end{proof}
  \begin{lemma}\thlabel{lem:4.6}
   Let $(\mathcal{A},\mathcal{B})$ be a maximal pair, then $|\mathcal{B}|\leq 2^{\lfloor\frac{n}{2}\rfloor}$
  \end{lemma}
  \begin{proof}
    It is a known result (see \cite{vanLint1999}) that for a linear code $C \subseteq \mathbb{F}^{n}_2$ and its dual code $C^{\perp}$,
 \begin{equation}\label{eq:6}
  \text{dim}(C) + \text{dim}(C^{\perp}) = n
 \end{equation}
 
 For any self-orthogonal code $C$, $C \subseteq C^{\perp}$. So,
 \begin{equation*}
  \text{dim}(C) \le \text{dim}(C^{\perp})
 \end{equation*}
  Applying equation \eqref{eq:6} in this inequality, we get
  \begin{gather*}
   n = \text{dim}(C) + \text{dim}(C^{\perp}) \ge 2 \text{dim}(C)\\
   \text{Therefore, } \text{dim}(C) \le \frac{n}{2}
  \end{gather*}
Since $\mathcal{B}$ is a self-orthogonal code (\thref{lem:4.5}), we get dim($\mathcal{B}$) $\le \frac{n}{2}$. Hence,
\begin{equation*}
 |\mathcal{B}| \le 2^{\lfloor \frac{n}{2} \rfloor}
\end{equation*}
  \end{proof}
  \begin{proposition}\thlabel{lem:4.7}
  If a set $A$ bisects $B_1$, $B_2$ and $B_1 \Delta B_2$, then $A$ also bisects $B_1 \cap B_2$.
  \end{proposition}
  \begin{proof}
  \begin{eqnarray*}
  |A \cap (B_1 \vartriangle B_2)| & = & \frac{|B_1 \vartriangle B_2|}{2} \text{ [A bisects $B_1 \Delta B_2$]} \\
  \Rightarrow |A \cap ((B_1 \backslash B_2) \cup (B_2 \backslash B_1))| & = & \frac{|B_1| + |B_2| - 2|B_1 \cap B_2|}{2}\\
  \Rightarrow |A \cap (B_1 \backslash B_2)| + |A \cap (B_2 \backslash B_1)| & = & \frac{|B_1|}{2} + \frac{|B_2|}{2} -|B_1 \cap B_2|
  \end{eqnarray*}
  $$
  \Rightarrow |A \cap B_1| - |A \cap (B_1 \cap B_2)| + |A \cap (B_2)| - |A \cap (B_1 \cap B_2)|  =  \frac{|B_1|}{2} + \frac{|B_2|}{2} -|B_1 \cap B_2|
 $$
\begin{eqnarray*}
  \Rightarrow \frac{|B_1|}{2} + \frac{|B_2|}{2} - 2|A \cap (B_1 \cap B_2)| & = & \frac{|B_1|}{2} + \frac{|B_2|}{2} -|B_1 \cap B_2| \\
  & & \text{ [since $A$ bisects both $B_1$ and $B_2$]} \\
  \Rightarrow 2|A \cap (B_1 \cap B_2)| & = & |B_1 \cap B_2| \\
  \Rightarrow |A \cap (B_1 \cap B_2)| & = & \frac{|B_1 \cap B_2|}{2}
 \end{eqnarray*}
  \end{proof}
  \begin{proposition}\thlabel{cor:4.8}
   $\mathcal{B}$ is closed under intersection.
  \end{proposition}
  \begin{proof}
    Let $B_{1}$,$B_{2}$ $\in \mathcal{B}$. We show that $B_{1}\cap B_{2} \in \mathcal{B}$. By \thref{lem:4.3}, $b_{1}+b_{2} \in \mathcal{B}$ i.e.,  $B_1 \Delta B_2 \in \mathcal{B}$. Let $A$ be any arbitrary member of $\mathcal{A}$.
    Now, $A$ bisects $ B_{1}, B_{2}$ and $B_1 \Delta B_2$ as $(\mathcal{A},\mathcal{B})$ is a cross bisecting pair. By \thref{lem:4.7}, $A$ bisects $B_1 \cap B_2$. Since $(\mathcal{A},\mathcal{B})$ is a maximal pair, we conclude that $B_1 \cap B_2 \in \mathcal{B} $. 
  \end{proof}
  Now, we prove the main result of this section,\thref{thm:1.2}, the characterization of maximal pairs.\\
  
 \textbf{Statement of \thref{thm:1.2}: }
 \emph{Let $(\mathcal{A},\mathcal{B})$ be a $\frac{1}{2}$-cross intersecting pair of families of subsets of $[n]$ with $|\mathcal{A}||\mathcal{B}|=2^n$. Then $(\mathcal{A},\mathcal{B})$ is one of the following $\lfloor\frac{n}{2}\rfloor +1$ pairs of families $(\mathcal{A}_{k},\mathcal{B}_{k})$, $0 \leq k \leq \lfloor \frac{n}{2} \rfloor$, up to isomorphism.
 \begin{center}
     $\mathcal{A}_{0}=2^{[n]}$ and $\mathcal{B}_{0}=\{\emptyset\}$
 \end{center}
 \begin{center}
     $\mathcal{A}_{k}=\{ A \in 2^{[n]} : |A \cap \{ 2i-1, 2i\}|= 1 ~~\forall i, 1\leq i\leq k\}$
 \end{center}
 \begin{center}
     $\mathcal{B}_{k}=\{ B \in 2^{[n]} : |B \cap \{ 2i-1, 2i \}| \in \{ 0, 2\}~~ \forall i, 1\leq i\leq k $ and $\forall j > 2k$, $j \notin B \}$, 
 \end{center}
 where $1 \leq k \leq \lfloor \frac{n}{2} \rfloor$. 
}
  
  \par
  By isomorphism, it is meant that for any maximal pair $(\mathcal{A},\mathcal{B})$, $\exists$ a bijective mapping
      $f:[n]\rightarrow [n]$
  such that if every $A \in \mathcal{A}$ is replaced by $A_{f}=\{ f(i)|i \in A\}$ and every $B \in \mathcal{B}$ is replaced by $B_{f}=\{ f(i)| i \in B\}$ then the families $(\mathcal{A}_{f},\mathcal{B}_{f})$, where $\mathcal{A}_{f}=\{A_{f}|A\in \mathcal{A}\}$ and $\mathcal{B}_{f}=\{B_{f}|B\in \mathcal{B}\}$, is a maximal pair which is one of $(\mathcal{A}_{k},\mathcal{B}_{k})$ , $0 \leq k \leq \lfloor \frac{n}{2} \rfloor  $.\par
  \begin{proof}
    Consider a maximal pair $(\mathcal{A}, \mathcal{B})$ where $\mathcal{B} \neq \{ \emptyset \}$. We write the elements of $\mathcal{B}$ as rows of a $0-1$ matrix $M_0$. Suppose $n_{0}$ columns have only $0$ entries in all the rows($n_{0}$ may be $0$). As the characterization is up to isomorphism, we may assume that these are the rightmost  $n_{0}$ columns of the matrix $M_0$.
    In each of the remaining  $n-n_{0}$ columns, from \thref{lem:3.1}, there are exactly $\frac{|\mathcal{B}|}{2}$ $1$'s and $\frac{|\mathcal{B}|}{2}$ $0$'s as $\mathcal{B}$ is a linear code. (by \thref{lem:4.3})\\
    Define 
    \begin{equation*} 
  B_1 = \bigcap_{\substack{1 \in B,\\ B \in \mathcal{B}}} B
  \end{equation*}
  We write the $\frac{|\mathcal{B}|}{2}$ rows containing 1 in the leftmost column of $M_0$ as the top $\frac{|\mathcal{B}|}{2}$ rows to obtain a new matrix $M_1$ from $M_0$. And $B_{1}$ is one of these rows according to \thref{cor:4.8}.
  Moreover, as all intersections are of even cardinality (\thref{lem:4.5}), $|B_{1}|$ is even. \\
  Let $|B_{1}|=2i_{1}$, $i_{1}\geq 1$.
  So, there are $2i_{1}-1$ elements in $B_{1}$ other than the element 1. Due to isomorphism, we may assume them to be $2,3,\ldots,2i_{1}$.\\
  If $2i_{1}+1\leq n-n_{0}$, then define the set $B_2$ as:
      \begin{equation*} 
  B_2 = \bigcap_{\substack{ 2i_{1}+1 \in B,\\ B \in \mathcal{B}}} B
  \end{equation*}
  \begin{claim} \thlabel{lem:4.ten}
    $1 \notin B_2$
  \end{claim} 
  \begin{proof}
     Assume for the sake of contradiction, $1 \in B_2$. This implies that for all the $\frac{|\mathcal{B}|}{2}$ sets which contain the element $2i_{1}+1$ also contain the element $1$. From \thref{lem:3.1}, (number of sets in $\mathcal{B}$ that contain the element $1$) = (number of sets in $\mathcal{B}$ that contain the element $2i_{1}+1$) = $\frac{|\mathcal{B}|}{2}$. Hence, for any $B \in \mathcal{B}$, $1 \in B$ $\Longleftrightarrow$ $2i_{1}+1 \in B$. This implies that $2i_{1}+1 \in B_{1}$, which is a contradiction. Hence, $1 \notin B_{2}$ and therefore $B_{2}$ does not belong to the top $\frac{|\mathcal{B}|}{2}$ rows of $M_{1}$.
  \end{proof}
  
  \begin{claim} \thlabel{lem:4.eleven}
    $B_1 \cap B_2 = \emptyset $
  \end{claim}
  \begin{proof}
     Assume for the sake of contradiction, $x \in B_1 \cap B_2$. Then $x$ is present in the $\frac{|\mathcal{B}|}{2}$ rows of the matrix $M_1$ whose intersection yields $B_1$. Since $x\in B_2 $ and  $B_2$ does not belong to these $\frac{|\mathcal{B}|}{2}$ rows of $M_1$ (by \thref{lem:4.ten}). Thus, we have the element $x$ present in at least $\frac{|\mathcal{B}|}{2} + 1$ rows of $M_1$, contradicting \thref{lem:3.1}. 
  \end{proof}
 We take the rows corresponding to the sets containing the $(2i_{1}+1)^{th}$ element that are not among the first $\frac{|\mathcal{B}|}{2}$ rows in $M_1$ and arrange them below
 the top $\frac{|\mathcal{B}|}{2}$ rows to create a matrix called $M_2$ from $M_1$. Again from \thref{lem:4.5}, $|B_2|$ is even, say $2i_2$. Due to isomorphism and \thref{lem:4.eleven}, we may assume that $2i_{1}+1,$\ldots$,2i_{1}+2i_{2}$ are these $2i_{2}$ elements.\par
 If $2i_{1}+2i_{2}+1 \leq n - n_{0}$, then define,
 \begin{equation*} 
  B_3 = \bigcap_{\substack{ 2i_{1} + 2i_{2} + 1 \in B,\\ B \in \mathcal{B}}} B
  \end{equation*}
 \begin{claim}\thlabel{lem:4.twelve}
   $1 \notin B_3$ and $2i_{1}+1 \notin B_3$.
 \end{claim}
 The proof is similar to that of \thref{lem:4.ten}
 \begin{claim}\thlabel{lem:4.thirteen}
   $B_1 \cap B_3 = \emptyset$ and $B_2 \cap B_3 = \emptyset$.
 \end{claim}
 The proof is again similar to that of \thref{lem:4.eleven}.\par
 We take the rows corresponding to the sets containing the $(2i_{1} +
 2i_{2} + 1) ^{th}$ element that are not among the first $r$ rows $(r >
 \frac{|\mathcal{B}|}{2})$ in $M_2$ which contain the elements $1$ or
 $2i_{1}+1$ and arrange them below the top $r$ rows of $M_2$ to create a
 matrix called $M_3$ from $M_2$. From \thref{lem:4.5} and the definition
 of $B_3$, we have $|B_3|=2i_3$, $i_3 \geq 1$. Due to isomorphism and
 \thref{lem:4.thirteen}, we may assume that
 $2i_{1}+2i_{2}+1,$\ldots$,2i_1+2i_2+2i_3$ are these $2i_3$ elements.
 \par
 We continue in this manner for $k$ steps by constructing sets
 $B_1, \ldots, B_k$ and matrices $M_1, \ldots ,M_k$, where $k \geq 1$, until we have $2i_1 + \cdots + 2i_k = n - n_0$.
 Observe that $B_1,\ldots, B_k$ and $P=\{ n-n_0+1,\ldots,n\}$ is a partition of $[n]$.
 
 \begin{center}
    \includegraphics[width=0.4\linewidth]{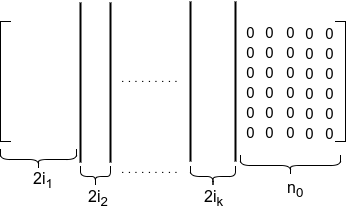}\label{fig:1}
    \captionof{figure}{Partitioning the universe and thereby the columns of $M_k$}
  \end{center}

 \begin{claim}\thlabel{lem:4.fourteen}
   For any set $B \in \mathcal{B}$, $j\in [k]$, we have $B \cap B_j \in \{ \emptyset, B_j \}$. Further, $B \cap P = \emptyset$.
 \end{claim}
 \begin{proof}
    From the definition of $P$, we have $B \cap P = \emptyset$. Let $j \in [k]$. Since $B_{j}$ is equal to the intersection of some $\frac{|\mathcal{B}|}{2}$ sets in $\mathcal{B}$, we have $B_j$ present as a subset of all these $\frac{|\mathcal{B}|}{2}$ sets. Applying \thref{lem:3.1}, we can say that no element of $B_j$ is present in any set in $\mathcal{B}$ other than these $\frac{|\mathcal{B}|}{2}$ sets. Hence, the claim.    
 \end{proof}
 
 From \thref{lem:4.fourteen}, observe that $ S=\{ B_{1},\ldots,B_{k}\}$ forms a basis of the row space of the matrix $M_k$.
The advantage of such a
    ``disjoint basis" is that the bisection in one part is independent of another.

\begin{center}
    \includegraphics[width=0.4\linewidth]{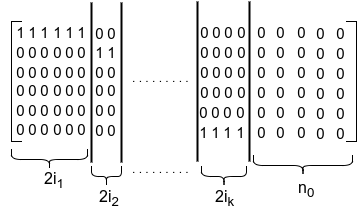}\label{fig:2}
    \captionof{figure}{Basis for the code $\mathcal{B}$}
  \end{center}

\begin{claim}\thlabel{lem:4.fifteen}
 A set $A \in \mathcal{A}$ bisects every set in $\mathcal{B}$ if and only if it bisects every set in  the basis $S$ of $\mathcal{B}$.
\end{claim}
\begin{proof}
   The forward direction is straightforward as $S \subseteq \mathcal{B}$. For the opposite direction, let $A \in \mathcal{A}$ be a set that bisects every member of $S$. Since the sets corresponding to the members in $S$ are disjoint, any $B \in \mathcal{B}$ can be written as a union of some of these sets.\\ Let $B= B_1 \cup \cdots \cup B_l$, where $\{B_1,\ldots, B_l\}\subseteq S$. Then,
 
  $|A\cap B|=|A \cap (\bigcup\limits_{j=1}^{l} B_j)|  = \sum\limits_{j=1}^{l} |A \cap B_j|
   = \sum\limits_{j=1}^{l} \frac{|B_j|}{2}
  = \frac{|\bigcup\limits_{j=1}^{l} B_j|}{2}
  = \frac{|B|}{2}$
 
\end{proof}

 Since each set $A \in \mathcal{A}$ bisects the sets $B_1,\ldots ,B_k$ and $P$, from \thref{lem:4.fifteen}, the set $A$ may contain any of the $2^{n_0}$ subsets of $P$, and $|A\cap B_1|=i_1,\ldots,|A\cap B_k|=i_k$. Since dim$(\mathcal{B})=k$, by \thref{lem:4.3}, we have  $|\mathcal{B}| = 2^k$.

\begin{align} 
 |\mathcal{A}||\mathcal{B}| & = \Big ( 2^{n_0} \cdot \prod\limits_{j=1}^{k} \binom{2i_j}{i_j} \Big) \cdot 2^k  \label{eq:7}
\end{align}
Recall that $\sum\limits_{j=1}^{k} 2i_j=n-n_0$. Right hand side of Equation  \eqref{eq:7}, is equal to $2^n$ if and only if $i_j=1$, $\forall j \in [k]$.

Thus, if $\mathcal{B}\neq \{\emptyset\}$, then $(\mathcal{A}_k,\mathcal{B}_k),$ $k \geq 1$, defined in the statement of the theorem are the only maximal pairs. This completes the proof of \thref{thm:1.2}. 

  \end{proof}

\end{section}
\begin{section}{Tight upper bound on $M_{\frac{c}{d}}(n)$ when $\mathcal{B}$ is $k$-uniform and characterization of the cases when the bound is achieved }
Let $(\mathcal{A},\mathcal{B})$ be a $\frac{c}{d}$ cross-intersecting pair of families of subsets of $[n]$, where $\frac{c}{d} \in [0,1]$ is an irreducible fraction. In this section, we deal with the scenario when $\mathcal{B}$ is $k$-uniform, where $0 < k \leq n$.
Since $\mathcal{B}$ is $k$-uniform, for any $A\in \mathcal{A}$ and any $B \in \mathcal{B}$, $|A \cap B|= \frac{ck}{d}=l$. Since $c$ is relatively prime with $d$, and $|A\cap B|$ is an integer, we have $k$ divisible by $d$. Therefore, we have a uniformly cross intersecting pair of families.\par
Alon and Lubetzky in \cite{Alon2009} found a tight upper bound for the case of uniformly cross intersecting families and fully characterized
the cases when the bound is achieved in the following theorem:
\begin{theorem}\thlabel{thm:5.1} [Theorem 1.1 in \cite{Alon2009}]
 There exists some $l_{0}>0$ such that, for all $l \geq l_{0}$, every $l$-cross intersecting pair $\mathcal{A},\mathcal{B} \subset 2^{[n]}$ satisfies: 
 \begin{center}
     $|\mathcal{A}||\mathcal{B}|\leq \binom{2l}{l}2^{n-2l}$
 \end{center}
 Furthermore, if $|\mathcal{A}||\mathcal{B}|= \binom{2l}{l}2^{n-2l}$, then there exists some choice of parameters $\kappa,\tau,n^{'}$:
 \begin{center}
     $\kappa \in \{ 2l-1, 2l \}, \tau \in \{0,\cdots,\kappa\}$\\
     $\kappa+\tau \leq n^{'} \leq n$
 \end{center}
 such that upto a relabelling of the elements of $[n]$ and swapping $\mathcal{A},\mathcal{B}$, the following holds:
 \begin{center}
 $\mathcal{A}=\{ \bigcup\limits_{T\in J} \, {T} : J \subset \{
 \{1,\kappa+1\},\cdots,\{\tau, \kappa+\tau\},
 \{\tau+1\},\cdots,\{\kappa\}\},|J|=l\}\:\times\:2^{X}$,\\
 $\mathcal{B}=\{L\cup\{\tau+1,\cdots, \kappa\}:L\subset
 \{1,\cdots,\tau,\kappa+1,\cdots,\kappa+\tau\}, |L\cap\{i,\kappa+i\}|=1$
 for all $i \in [\tau] \}\:\times\:2^{Y}$
 \end{center}
 where $X=\{ \kappa + \tau +1, \cdots, n^{'}\}$ and $Y=\{n^{'}+1,\cdots,n\}.$
\end{theorem}

Let $(\mathcal{A},\mathcal{B})$ be a $\frac{c}{d}$ cross-intersecting
family where $\mathcal{B}$ is $k$-uniform. From \thref{thm:5.1}, there
exists a $k_{0}>0$ such that if $\frac{ck}{d}=l>k_{0}$, then
$|\mathcal{A}||\mathcal{B}|\leq \binom{2l}{l}2^{n-2l}$. Consider the case
when $\mathcal{B}$ corresponds to $\mathcal{B}$ of \thref{thm:5.1}. If
$|\mathcal{A}||\mathcal{B}| = \binom{2l}{l}2^{n-2l}$, then $n'= n$,
$Y=\emptyset$, and $k=\kappa$ in the statement of \thref{thm:5.1}.
Since $l=\frac{ck}{d}$ and $k \in \{ \frac{2ck}{d}-1,\frac{2ck}{d}\}$, we have the following two cases:\\\\
\textbf{Case 1:} $k=\frac{2ck}{d}-1$. Then, $(k+1)d=2ck$. Since $gcd(c,d)=1$ and $gcd(k,k+1)=1$, we have $k|d|2k$. Thus, $d=k$ or $d=2k$. We claim that $d=2k$ is an invalid case. This is because, when $d=2k$, we have $c=k+1$. Since $gcd(c,d)=1$, $k$ cannot be odd. And if $k$  is even, then $l=\frac{ck}{d}=\frac{k+1}{2}$ is not an integer. So, the only valid case is $d=k$, $c=\frac{k+1}{2}=l$ and $k$ is an odd integer.\\
\textbf{Case 2:} $k=\frac{2ck}{d}$. Then, $\frac{c}{d}=\frac{1}{2}$, that is $(\mathcal{A},\mathcal{B})$ is a cross bisecting pair.  Since $l=\frac{ck}{d}=\frac{k}{2}$ is an integer, $k$ must be even in this case.\\
If $\mathcal{B}$ corresponds to $\mathcal{A}$ of \thref{thm:5.1}, we have $X=\emptyset$, $\tau=0$, $\mathcal{B}$ is $k(=l)$-uniform, $l=\frac{ck}{d}$. Thus, we have $\frac{c}{d}=1$, $\mathcal{A}=\{\{1,\ldots,\kappa\}\}\times2^Y$ where $Y=\{\kappa+1,\ldots,n\}$ and $\mathcal{B}=\binom{[\kappa]}{k}$, $\kappa \in \{ 2k-1, 2k\}$ up to a relabelling of the elements.\\

This leads us to the main result of this section. \vspace{0.2in} \\
\textbf{Statement of \thref{thm:1.3}:}
\itshape{Let $(\mathcal{A}$,$\mathcal{B})$ be a $\frac{c}{d}$-cross intersecting pair of families of subsets of $[n]$. Let $\mathcal{B}$ be $k$-uniform. Then, there exists some $k_{0}>0$, such that for $k>k_{0}$ we have 
 \begin{center}
     $|\mathcal{A}||\mathcal{B}|\leq \binom{\frac{2ck}{d}}{\frac{ck}{d}}2^{n-\frac{2ck}{d}}$
 \end{center}
 and the bound is tight if and only if, either $(a)$ or $(b)$ hold:
 \begin{center}
 \begin{enumerate}[(a)]
     \item  When $\frac{c}{d}=1$, $\mathcal{A}=\{\{1,\ldots,\kappa\}\}\times2^Y$, $\mathcal{B}=\binom{[\kappa]}{k}$ where $Y=\{\kappa+1,\ldots,n\}$ and $\kappa \in \{ 2k-1, 2k\}$ up to a relabelling of the elements of $[n]$.

 \item When $\frac{c}{d}\neq 1$:
 
     \begin{enumerate}[(i)]
     \item If $k$ is even, $c=1$, $d=2$, $\frac{ck}{d}=\lceil\frac{k}{2}\rceil$,
     \item If $k$ is odd, $c=\frac{k+1}{2}$, $d=k$, $\frac{ck}{d}=\lceil\frac{k}{2}\rceil$, 
      \end{enumerate}
\end{enumerate}
 \end{center}
 and for both the cases($(i)$ and $(ii)$), 
 there exists some $\tau$ such that, 
 $k+\tau\leq n$ 
and up to a relabelling of the elements of $[n]$,
\begin{center}
    $\mathcal{A}=\{ \cup_{T\in J} \, {T} : J \subset \{
    \{1,k+1\},\ldots,\{\tau, k+\tau\},
    \{\tau+1\},\ldots,\{k\}\},|J|=\lceil\frac{k}{2}\rceil\}\times2^{X}$
\end{center}
where $X=\{k+\tau+1,\ldots,n\}$ and
\begin{center}
$\mathcal{B}=\{L\cup\{\tau+1,\ldots, k\}:L\subset \{1,\ldots,\tau,k+1,\ldots,k+\tau\}, |L\cap\{i,k+i\}|=1$ for all $i \in [\tau] \}$.
    
\end{center}
}

\end{section}

\begin{section}{Discussion}
What are those pairs of $\frac{c}{d}$-cross intersecting families $(\mathcal{A},\mathcal{B})$ which achieve $|\mathcal{A}||\mathcal{B}|=2^n$ (equal to the upper bound for $\mathcal{M}_{\frac{c}{d}}(n)$ proved in \thref{thm:1.1})? In the introduction we characterize such families when $\frac{c}{d}=0$ and $\frac{c}{d}=1$. In \thref{thm:1.2}, we characterize such families when $\frac{c}{d}=\frac{1}{2}$. From \thref{thm:1.3}, we see that when $\mathcal{B}$ is $k$-uniform, $|\mathcal{A}||\mathcal{B}|$ is maximized when $\frac{c}{d}$ is $1$ or nearly $\frac{1}{2}$($\frac{1}{2}$ or $\frac{1}{2}+\frac{1}{2k}$). For $\frac{c}{d} \in (0,1)$, besides the case $\mathcal{A}=2^{[n]}$, $\mathcal{B}=\{\emptyset\}$, is $|\mathcal{A}||\mathcal{B}|=2^n$ achieved only when $\frac{c}{d}$ is close to $\frac{1}{2}$?

\end{section}

\bibliographystyle{ieeetr}

\end{document}